\documentclass[11pt,]{amsart}
\usepackage{amsmath}
\usepackage{amsxtra}
\usepackage{amscd}
\usepackage{amsthm}
\usepackage{amsfonts}
\usepackage{amssymb}
\usepackage{eucal}
\usepackage{epsfig}
\usepackage{graphics}
\usepackage{accents}

\usepackage{mathtools}
\usepackage{etex}
\usepackage{tikz}
\usetikzlibrary{decorations.markings,shapes,matrix,calc}
\usepackage{here} 
\numberwithin{equation}{section}
\newtheorem{thm}{Theorem}[section]
\newtheorem{prop}[thm]{Proposition}
\newtheorem{lem}[thm]{Lemma}

\newtheorem{cor}[thm]{Corollary}

\newcommand{\nn}{\nonumber}

\newcommand{\ds}[1]{\displaystyle #1}

\newcommand{\C}{{\mathbb C}}
\newcommand{\Z}{{\mathbb Z}}

\newcommand{\E}{{\mathcal E}}

\newcommand{\cH}{\mathcal H}

\newcommand{\uh}{\hat{u}}
\newcommand{\vh}{\hat{v}}
\newcommand{\Uh}{\hat{U}}
\newcommand{\Vh}{\hat{V}}
\newcommand{\hh}{\hat{h}}
\newcommand{\GS}{\mathfrak{S}}
\newcommand{\gl}{\mathfrak{gl}}

\newcommand{\slt}{\mathfrak{sl}_2}

\newcommand{\ssh}{\textsf{h}}
\newcommand{\bph}{{\boldsymbol \varphi}}

\newcommand{\ssa}{\textsf{a}}
\newcommand{\ssq}{\textsf{q}}
\newcommand{\ssx}{\textsf{x}}

\newcommand{\End}{\mathop{\rm End}}

\newcommand{\Sym}{\mathrm{Sym}}

\newcommand{\wt}{{\rm wt}\,}

\newcommand{\al}{\alpha}

\newcommand{\gh}{\widehat{\mathfrak{g}}}
\newcommand{\g}{\mathfrak{g}}

\newcommand{\trr}{\tilde{r}}
\newcommand{\tu}{\tilde{u}}
\newcommand{\tv}{\tilde{v}}

\newcommand{\WD}{\mathcal{W}D(2,1;\alpha)}
\newcommand{\baa}{\mathbf{a}}

\newcommand{\brr}{\mathbf{r}}

\begin{document}
\begin{title}[Deformation of $\widehat{\mathfrak{sl}}_2$ coset VOA]
{Towards trigonometric deformation of $\widehat{\mathfrak{sl}}_2$ coset VOA}
\end{title}

\author{B. Feigin, M. Jimbo, and E. Mukhin}
\address{BF: National Research University Higher School of Economics, 
Russian Federation, International Laboratory of Representation Theory 
and \newline Mathematical Physics, Russia, Moscow,  101000,  
Myasnitskaya ul., 20 and Landau Institute for Theoretical Physics,
Russia, Chernogolovka, 142432, pr.Akademika Semenova, 1a
}
\email{bfeigin@gmail.com}
\address{MJ: Department of Mathematics,
Rikkyo University, Toshima-ku, Tokyo 171-8501, Japan}
\email{jimbomm@rikkyo.ac.jp}
\address{EM: Department of Mathematics,
Indiana University-Purdue University-Indianapolis,
402 N.Blackford St., LD 270,
Indianapolis, IN 46202, USA}\email{emukhin@iupui.edu}
\begin{abstract} 
We discuss the quantization of the $\widehat{\mathfrak{sl}}_2$ coset vertex operator algebra 
$\WD$
using the bosonization technique.
We show that after quantization there exist three families of commuting integrals of motion coming from three copies of the quantum toroidal algebra associated to 
${\mathfrak{gl}}_2$.
\end{abstract}

\date{\today}

\maketitle 

\bigskip

\section{Introduction}\label{sec:intro}
In this paper we study the trigonometric deformation of the coset algebra 
in the case of $\slt$.

\medskip

Let $\g$ be a finite-dimensional semisimple Lie algebra. Fix three levels 
$k_1,k_2,k_3$ such that 
\begin{align}\label{sum} k_1+k_2+k_3=-2h^\vee, \end{align}
 where $h^\vee$ is the dual Coxeter number of $\g$. By definition, the 
vacuum representation of the coset algebra $C(\g;k_1,k_2,k_3)$ is the 
space of semi-infinite cohomology $H^{\infty/2}(\gh,\g,V_{k_1}\otimes V_{k_2}\otimes V_{k_3})$, where $V_{k_i}$ are the 
vacuum representations of $\gh$ of 
level $k_i$. In this paper we consider only generic $k_1,k_2,k_3$ 
satisfying \eqref{sum}, 
when the semi-infinite cohomology is non-trivial only in one dimension.

The above definition is symmetric in $k_1,k_2,k_3$. A more standard 
definition of the coset algebra is non-symmetric. Namely, one defines $C(\g;k_1,k_2,k_3)$ to be the space of invariants 
$(V_{k_1}\otimes V_{k_2})^{\g\otimes \C[t]}$. 
Informally, the coset algebra is the commutant of diagonal $\gh$ 
inside $U\gh_{k_1}\otimes U\gh_{k_2}$,  
where $U\gh_{k_i}$
is the universal enveloping algebra of $\gh$ on level $k_i$.

There is yet another definition of the coset algebra $C(\g;k_1,k_2,k_3)$ as a subalgebra of $V_{k_1}$.
For example, in the limit $k_2\to \infty$, the coset algebra $C(\g; k_1,\infty,\infty)$ is identified with the subalgebra of invariants $(V_{k_1})^\g$. For general $k_1,k_2,k_3$ one needs to use the machinery of the screenings. This is the construction used in this paper, so let us describe the main points.

\medskip

Quite generally, given a vertex operator algebra (VOA) $V$ and a set of fields $s_i(z)$, $i=1,\dots,n$, called screening currents, one defines screenings $S_i=\int s_i(z) dz$ and the algebra $V(S_1,\dots,S_n)$ 
consisting of
all local fields which commute with all screenings: 
$V(S_1,\dots,S_n)=\{a(z)\ |\ [a(z),S_i]=0,\ i=1,\dots,n\}$. The algebra $V(S_1,\dots,S_n)$ is closed under 
the operator product, and if it contains a Virasoro current, then it is a VOA.

Let us assume that $\g$ is simply laced.
Let $\widehat{\mathfrak{h}}\subset\gh$ be the Heisenberg algebra generated by Cartan currents and let $H$ be a lattice VOA of $\widehat{\mathfrak{h}}$. Let $s_i(z)=\exp(c_+\al_i(z))$, $i=1,\dots, n$, $n={\rm{rk}}(\g)$, $c_+\in\C^\times$, be the vertex operators corresponding to rescaled simple roots of $\g$. 
Then $H(S_1,\dots,S_n)=W(\g)$ is the W-algebra which can be alternatively obtained by the quantum Drinfeld-Sokolov reduction of 
$U\gh$ with respect to a maximal 
nilpotent subalgebra $\widehat{\mathfrak{n}}$, see \cite{FF1}. 
 Consider the dual screening currents $s_i^\vee(z)=\exp(-c_-\al_i(z))$,  and dual screenings 
$S_i^\vee=\int s_i^\vee(z) dz$, where $c_-\in\C^\times$ is another parameter.

Consider the algebra $\tilde V= V_{k_1}\otimes H$ with screening currents $\tilde s_i(z)=e_i(z)s_i(z)$, $\tilde s_i^\vee(z)=f_i(z)s_i^\vee(z)$  where $e_i(z)$, $f_i(z)$ are 
the  $\gh$ currents corresponding to the 
positive and negative simple root vectors.
Let $\tilde S_i=\int \tilde s_i(z) dz$, $\tilde S_i^\vee=\int \tilde s_i^\vee(z) dz$ be the corresponding screenings.
Then $[\tilde S_i,\tilde S_j^\vee]=0$ and 
$\tilde V (\tilde S_1,\dots, \tilde S_n,\tilde S_1^\vee,\dots, \tilde S_n^\vee)$ is the coset algebra $C(\g;k_1,k_2,k_3)$ together with the diagonal copy of the $\hat{\mathfrak{h}}$, where $c_+,c_-$ and $k_2,k_3$ are related by $(c_+)^2(k_2+h^\vee)=1$, $(c_-)^2(k_3+h^\vee)=-1$.

In particular, using the Wakimoto realization  of $V_{k_1}$, see  \cite{W}, \cite{FF}, we obtain a free field realization of the coset algebra.

 A similar, more involved, construction exists in the non-simply laced and in the supersymmetric cases.

\medskip

On mathematically rigorous level, the correspondence between three different definitions of the coset algebra is not fully established.
For the latest results, see \cite{ACL}.

\medskip

There exists a deep connection between VOAs and quantum groups. In particular, there exists a relationship between
the quasi-tensor category of representations of a VOA and the tensor category of representations of a quantum group. The screening construction is a reflection of this phenomenon. In general, the screenings which define the VOA generate (in some appropriate sense) the nilpotent part of the quantum group. In our situation, the coset algebra is a reduction of $V_{k_1}\otimes V_{k_2}\otimes V_{k_3}$, therefore we expect to have screenings which generate the nilpotent parts of three commuting quantum groups $U_{q^{(j)}}\g$ with $q^{(j)}=\exp(2\pi i/(k_j+h^\vee))$, $j=1,2,3$.

This is indeed so in the case of $\slt$, see \cite{FS} and Section \ref{FS}. 
The algebra 
$U(\widehat{\mathfrak{sl}}_2)_{k_1}$ has the 
Wakimoto realization through three
free fields. Then there are three
screening currents: 
$\tilde s_1(z)$, $\tilde s_1^\vee(z)$,  
and the Wakimoto screening current. 
These screening currents have the form $\theta_i(z)\exp(\tilde \theta_i(z))$, 
where $\theta_i(z), \tilde \theta_i(z)$ are free fields,  
and the corresponding screenings commute. We call these 
screenings ``bosonic". 

Given an embeddings of VOAs  $W_1\subset W_2$, 
one can ask whether it is possible to find all screenings: 
that is, to find all fields $s(z)$ in $W_2$ such that
$\int s(z)dz$ commutes with $W_1$. 
Unfortunately, we know no results 
in this direction. However, in the present situation, we
observe a dual ``fermionic picture". 
Namely, there are three more screening currents 
which anticommute,  
and the corresponding screenings form 
the 
nilpotent part of the 
quantum Lie superalgebra of type $D(2,1;\alpha)$. The fermionic currents are simple exponentials. In particular, we expect that the $\slt$ coset algebra can be obtained by the 
quantum Drinfeld-Sokolov reduction from $\hat{D}(2,1;\alpha)$.

\medskip

The next important ingredient is the description of integrable systems 
---the commutative subalgebras of vertex operator algebras. 
Such subalgebras are also defined as commutants of a set of screenings. 
In general
the set of local fields commuting with all screenings may be trivial, 
but under certain circumstances there exists an 
infinite series of descendants 
$a_j(z)$ of the vacuum
such that $I_j=\int a_j(z)dz$ commute with all screenings. 
The operators $I_j$ are called local integrals of motion. 
In all known cases, the set of screenings leading to an infinite set of local 
integrals of motion
corresponds to an affine quantum group. 

There exist also non-local integrals of motion given by suitable integrals of products of screening currents, see \cite{BLZ}.
Conjecturally, the cycles of integration defining the non-local integrals of motion
correspond to elements of the center of the affine quantum group at the critical level.

In the case of $\slt$ this is achieved by 
adjoining
one more fermionic screening, 
or three more bosonic screenings. 
Four fermionic screenings correspond to nilpotent part of 
the quantum affine superalgebra
$\hat{D}(2,1;\alpha)$, 
and three pairs of bosonic screenings correspond 
to the nilpotent parts of three commuting copies 
of the quantum affine algebra
$\widehat{\mathfrak{sl}}_2$. 
For the latest developments about the integrals of motion, 
see \cite{BL}, \cite{BKL}.

\medskip

In our previous papers \cite{FJMM}, \cite{FJMM1}, \cite{FJM}
we have studied the trigonometric deformation of W algebras $W\mathfrak{sl}_n$.
The advantage of deformation is that the non-local integrals of motion can be obtained from 
the R-matrix formalism using quantum toroidal algebras. In particular, the non-local 
integrals of motion 
can be described by appropriate explicit integrals, we can prove they commute and, moreover, 
we can study their spectrum by the machinery of Bethe ansatz. 
The aim of the present paper is to discuss
the trigonometric deformation of the picture mentioned above for the coset W algebra.

We use four free fields and a three dimensional lattice to construct the following objects.
\begin{itemize}
\item Four fermionic screening currents  which are simple exponentials. We call the corresponding screenings $\sigma_i$, 
$i=0,1,2,3$.  

\item Three pairs of bosonic screening currents which are sums of two exponential terms.The corresponding screenings $\rho_i,\tau_i$, $i=1,2,3$, anticommute with fermionic screenings and, in addition, satisfy $[\rho_i,\rho_j]=[\tau_i,\tau_j]=0$ for all $i,j$, and $[\rho_i,\tau_j]=0$ for $i\neq j$.

\item Three quantum affine algebras 
$U_{q^{(i)}}\widehat{\mathfrak{sl}}_2$ 
on level $q^{k_i}$, $i=1,2,3,$ in Wakimoto realization. 
Algebra $U_{q^{(i)}}\widehat{\mathfrak{sl}}_2$ 
commutes with one bosonic screening $\rho_i$ and two fermionic screenings $\sigma_j$, $j\neq i,0$. The screening currents 
corresponding to $\rho_j$, $j\neq i$ are 
obtained by dressing the $E(z)$ and $F(z)$ currents of  $U_q^{(i)}\widehat{\mathfrak{sl}}_2$ by 
an appropriate free field.

\item Three quantum toroidal $\gl_2$ algebras $\E_{2,i}(q^{-k_i-2},q^2,q^{k_i})$ on level 
$q^{k_i}$, $i=1,2,3$. The screening currents 
corresponding to $\rho_i$ and $\tau_i$ are obtained by dressing the currents 
$F_1(z)$ and $F_0(z)$ of $\E_{2,i}$.

\item Three sets of integrals of motion $\mathbb{G}_{i,M}^\nu$, $i=1,2,3$, $\nu=0,1$, $M\in\Z_{\geq 1}$ 
coming from $\E_{2,i}$. All integrals of motion commute among themselves.
\end{itemize}

We call the commutativity of the three sets of integrals the triality. 
The triality is very different from standard duality described in \cite{FJM2}. 
In some sense, the triality corresponds to the commutativity of three sets of integrals of motion coming from the same (unknown) algebra. We expect the existence of 
yet another set of integrals of motion coming 
from quantum toroidal $D(2,1;\alpha)$ 
which should be in the standard duality with the three sets of integrals described in this paper.
\medskip

There are several open questions: 
(1) to give an explicit description of the deformed coset algebra, that is,
to construct local currents which commute with three fermionic screenings; 
(2) to find non-local integrals of motion obtained from quantum toroidal $D(2,1;\alpha)$;
 and (3) to obtain local integrals of motion. We hope to address these issues
in some other occasion.
\medskip

The paper is constructed as follows. In Section \ref{FS} we recall some formulas from \cite{FS} in the non-deformed picture.
In Section \ref{deformed KS} we give the quantization of these formulas. We construct three actions of quantum affine $\slt$ and extend them to the action of quantum toroidal $\mathfrak{gl}_2$ in Section \ref{sec:toroidal gl2}. That allows us to show that the three integrable systems commute in Section \ref{sec:IM}. In Section \ref{sec:D21} we discuss the connection to quantum toroidal $D(2,1;\alpha)$. The Appendices contain some technical formulas and computations.

\section{W algebra $\WD$}\label{FS}

The W algebra $\WD$, termed also ``corner-brane $W$ algebra'' \cite{LZ}, 
has many intriguing features. From conformal field theory (CFT)
point of view, it can be viewed as a coset theory 
$(\mathfrak{sl}_2)_{k_1}\times
(\mathfrak{sl}_2)_{k_2}/(\mathfrak{sl}_2)_{k_1+k_2}$
with generic parameters $k_1,k_2$. 
In \cite{FS} it was identified with a quantum Hamiltonian reduction 
of the exceptional affine Lie superalgebra $\hat{D}(2,1;\alpha)$. 
Algebra $\WD$ governs the CFT limit of an integrable 
quantum field theory introduced in \cite{Fa}. 
In a remarkable work \cite{BL} the ODE/IM correspondence for this CFT 
model was found to be related to 
generalized hypergeometric Opers, and was further extended to include 
Fateev's massive integrable field theory.

The aim of this section is to 
summarize basic aspects of $\WD$ following \cite{FS}
and motivate the construction in the main text. 

Algebra $\WD$ is realized in the Fock space of three 
bosonic free fields
$\varphi_i(z)=Q_i+\varphi_{i,0}\log z-
\sum_{n\neq 0}\frac{\varphi_{i,n}}{n}z^{-n}$ ($i=1,2,3)$
satisfying 
$[\varphi_{i,0},Q_j]=\delta_{i,j}$, 
$[\varphi_{i,m},\varphi_{j,n}]=\delta_{i,j}n \delta_{m+n,0}$.
The total Fock space decomposes into subspaces
\begin{align*}
\bar\cH=\bigoplus_{\beta}\bar\cH_{\beta}\,,
\quad
 \bar\cH_{\beta}=\C[\varphi_{i,-n}\mid n>0,i=1,2,3]\otimes\C e^\beta\,,
\end{align*}
where the ``momentum'' $\beta$ runs over the three dimensional space
$\oplus_{i=1}^3\C Q_i$. We shall refer to $\bar\cH_{\beta}$'s as 
``sectors''.

Fix complex numbers $k_i$ ($i=1,2,3$) satisfying
$k_1+k_2+k_3=-4$.
Choose vectors $\baa_i\in\C^3$ ($i=1,2,3$) with scalar product 
\begin{align*}
&\baa_i\cdot\baa_i=1\,,\quad
\baa_1\cdot\baa_2=k_3+1\quad \&\, cycl.. 
\end{align*}
Here and after we shall use the symbol ``$\& \, cycl.$''
to imply all equations 
obtained by cyclically permuting the indices $(1,2,3)$. 
Using the vector notation $\bph(z)=(\varphi_1(z),\varphi_2(z),\varphi_3(z))$, 
we introduce fermionic screening operators 
 $\sigma_i$ ($i=1,2,3$) by  
\begin{align}
&\sigma_i=\int :e^{\baa_i\cdot\bph(z)}:\, dz\,.
\label{CFTf-scr}
\end{align}

Next let 
$\brr_1=-\frac{1}{k_1+2}(\baa_2+\baa_3)\quad \&\ cycl.$.  
These vectors are mutually orthogonal. 
We introduce bosonic screening operators $\rho_i$ ($i=1,2,3$) by  
\begin{align}
\rho_1=\int :(\baa_3\!\cdot\!\partial\bph(z))
e^{\brr_1\cdot\bph(z)}: \,dz\quad
\&\ cycl..
\label{CFTb-scr}
\end{align}

Now let  $\alpha=k_1/(k_3+2)$.  
The W algebra $\WD$ is a vertex subalgebra in the vacuum sector $\bar\cH_{0}$. 
It is defined either as the intersection of the kernels of $\sigma_i$, 
or that of the kernels of $\rho_i$ \cite{FS},  
\begin{align*}
\WD=\bigcap_{i=1}^3\mathrm{Ker}\,\Bigl(\sigma_i\bigl|_{\bar\cH_0}\Bigr)
=\bigcap_{i=1}^3\mathrm{Ker}\,\Bigl(\rho_i\bigl|_{\bar\cH_0}\Bigr)
\quad \subset  \bar\cH_{0}\,.
\end{align*}
As is well known, W algebras associated with simple Lie algebras
admit two descriptions in terms of 
quantum groups which are Langlands dual to each other. 
In the present case the relevant quantum groups look rather different.
The fermionic screening operators $\{\sigma_i\}_{i=1,2,3}$ 
generate  the nilpotent part $U^+_\ssq D(2,1;\alpha)$ 
of the quantum group
associated with the exceptional Lie superalgebra
$D(2,1;\alpha)$ (whence the notation $\WD$).
In turn, the bosonic screening operators
$\{\rho_i\}_{i=1,2,3}$ generate 
$U_{\ssq_1}^+\mathfrak{sl}_2\otimes U_{\ssq_2}^+\mathfrak{sl}_2\otimes 
U_{\ssq_3}^+\mathfrak{sl}_2$ 
where $\ssq_i=\exp\bigl(2\pi i/(k_i+2)\bigr)$.

It is easy to see that $\WD$ contains the Virasoro current
\begin{align*}
T_2(z)=\frac{1}{2}\sum_{i=1}^3\frac{k_i+2}{2} :\partial r_i(z)^2:+ 
\frac{1}{2}\sum_{i=1}^3 \partial^2 r_i(z)\,,
\quad r_i(z)=\brr_i\cdot\bph(z)\,,
\end{align*}
with central charge $c=3-\sum_{i=1}^3\frac{6}{k_i+2}$.
As vertex algebra, $\WD$ is conjecturally  
generated by $T_2(z)$ and a  current $T_4(z)$
of spin $4$. 
An explicit expression for $T_4(z)$ can be found in \cite{FS}, \cite{LZ}. 

In order to discuss integrals of motion, 
we need additional screening operators.
Let $\sigma_0$ be given by the same formula \eqref{CFTf-scr} 
with $\baa_0=-\sum_{i=1}^3\baa_i$. 
Let further 
\begin{align*}
\tau_i=\int :(\baa_i\!\cdot\!\partial\bph(z))
e^{-\brr_i\cdot\bph(z)}: \,dz\,
\end{align*}
for $i=1,2,3$. Then $\{\sigma_i\}_{i=0,1,2,3}$ generate
 $U^+_\ssq \hat{D}(2,1;\alpha)$, and 
 $\{\rho_i,\tau_i\}_{i=1,2,3}$ generate 
$U_{\ssq_1}^+\widehat{\mathfrak{sl}}_2\otimes 
U_{\ssq_2}^+\widehat{\mathfrak{sl}}_2\otimes 
U_{\ssq_3}^+\widehat{\mathfrak{sl}}_2$. 
One expects that their joint kernels coincide,
\begin{align*}
\bigcap_{i=0}^3\mathrm{Ker}\,\bigl(\sigma_i\bigl|_{\bar\cH_0}\bigr)
=\bigcap_{i=1}^3\Bigl(
\mathrm{Ker}\,\bigl(\rho_i\bigl|_{\bar\cH_0}\bigr)
\cap \mathrm{Ker}\,\bigl(\tau_i\bigl|_{\bar\cH_0}\bigr)\Bigr)\,,
\end{align*}
giving a commutative algebra generated by 
integrals of local densities 
\begin{align*}
\mathbb{I}_n=\int P_{n+1}(z)\,dz\quad (n=1,3,5,\ldots).
\end{align*}
They are called ``local'' integrals of motion. 
One expects also the existence of an infinite family of
``non-local'' integrals of motion given by
multiple integrals of products of 
screening currents. The bosonic ones  associated with 
$U_{\ssq_1}^+\widehat{\mathfrak{sl}}_2\otimes 
U_{\ssq_2}^+\widehat{\mathfrak{sl}}_2\otimes 
U_{\ssq_3}^+\widehat{\mathfrak{sl}}_2$ 
are discussed in \cite{BL}. 
Little is known about non-local integrals of motion associated with 
$U^+_{\ssq}\hat{D}(2,1;\alpha)$.

\vskip 2cm

\section{Deformed Screening Currents}\label{deformed KS}\quad
In this section, we $q$-deform the screening currents 
considered in the previous section. 
We fix $q\in\C^\times$ satisfying $|q|<1$, and use the standard symbol 
$[x]=(q^x-q^{-x})/(q-q^{-1})$.

Unlike the conformal case, we use here {\it four} free fields 
\begin{align*}
&a_i(z)=Q_{a_i}+a_{i,0}\log z-\sum_{n\neq 0}\frac{a_{i,n}}{n}z^{-n} \,,
\quad i=0,1,2,3.
\end{align*}
We choose $\{a_{i,n}\mid n\neq0, i=0,1,2,3\}$ to be an
independent set of oscillators. Concerning the zero modes, however, 
we assume linear relations 
\begin{align*}
&a_{0,0}+a_{1,0}+a_{2,0}+a_{3,0}=0\,,\quad
Q_{a_0}+Q_{a_1}+Q_{a_2}+Q_{a_3}=0\,.
\end{align*}
Hence the entire Fock space decomposes into sectors
\begin{align*}
 &\cH=\bigoplus_{\beta}\cH_\beta\,,
\quad \cH_{\beta}=\C[a_{i,-n}\mid n>0,i=0,1,2,3]\otimes \C e^\beta\,,
\end{align*}
with the momentum $\beta$ running over 
the same three dimensional space $\oplus_{i=1}^3\C Q_{a_i}$
as in the CFT case.

We set the following commutation relations among the Fourier modes,
\begin{align}
&[a_{i,n},a_{i,-n}]=n\,,\quad [a_{i,0}, Q_{a_i}]=1\,,
\label{aii}\\
&[a_{0,n},a_{i,-n}]=n\frac{[(k_i+1)n]}{[n]}\,,
\quad  [a_{0,0}, Q_{a_i}]=k_i+1\,,
\label{a01}
\\
&[a_{1,n},a_{2,-n}]=n\frac{[(k_3+1)n]}{[n]}\,,\quad 
 [a_{1,0}, Q_{a_2}]=k_3+1\quad \&\ cycl..
\label{a12}
\end{align}
All other commutators are zero.

Fermionic screening currents in the deformed case  
are still pure exponentials,
\begin{align}
&\sigma_i(z)=:e^{a_i(z)}:\,,
\quad i=0,1,2,3.
\label{sig-q}
\end{align}
Eq. \eqref{aii} ensures that the $\sigma_i(z)$'s are ordinary fermions. 
On the other hand, deformed bosonic screening currents
 $\rho_i(z)$, $\tau_i(z)$ ($i=1,2,3$) are sums of two exponential terms
\begin{align}
&\rho_i(z)=\frac{1}{(q-q^{-1})z}
\bigl(:e^{r^+_i(z)}:-:e^{r^-_i(z)}:\bigr)
\,,
\label{rho-q}\\
&\tau_i(z)=\frac{1}{(q-q^{-1})z}
\bigl(:e^{t^+_i(z)}:-:e^{t^-_i(z)}:\bigr)\,.
\label{tau-q}
\end{align}
The exponents are given by 
\begin{align*}
r^\pm_1(z)&=-a_2(k_1+2;z)-a_3(k_1+2;q^{\mp(k_1+2)}z)\quad \&\ cycl.,
\\
t^\pm_i(z)&
=-a_0(k_i+2;z)-a_i(k_i+2;q^{\mp(k_i+2)}z)\,,
\end{align*}
where 
\begin{align*}
a_i(N;z) =\frac{1}{N}(Q_{a_i}+a_{i,0}\log z)
-\sum_{n\neq0}\frac{[n]}{[N n]}\frac{a_{i,n}}{n}z^{-n}\,.
\end{align*}
It is easy to see that \eqref{sig-q}, \eqref{rho-q}, \eqref{tau-q} 
tend to the CFT counterparts in the limit $q\to1$.  
Screening operators are defined by
\begin{align*}
&\sigma_i=\int \sigma_i(z)\,dz\,,
\quad
\rho_j=\int \rho_j(z)\,dz\,, \quad 
\tau_j=\int \tau_j(z)\,dz\,. 
\end{align*}
These operators make sense on sectors $\cH_\beta$ with appropriate $\beta$. 
More precisely, 
$\sigma_i$ is well-defined if $[a_{i,0},\beta]\in\Z$, 
while $\rho_1,\tau_1$ are well-defined if $[a_{2,0}+a_{3,0},\beta]\in(k_1+2)\Z$. 

\begin{prop}\label{prop:q-screening}
The following relations hold true
\begin{align*}
 &[\sigma_i,\rho_j]_+=[\sigma_i,\tau_j]_+=0\quad (i=0,1,2,3,\ j=1,2,3)\,,\\
&[\rho_i,\rho_j]=[\rho_i,\tau_j]=[\tau_i,\tau_j]=0\, 
\quad (i,j=1,2,3, i\neq j)\,,
\end{align*}
on each sector $\cH_\beta$ such that the left hand sides are defined.
\end{prop}
\begin{proof}
We use contraction rules given in 
Table \ref{tab:art} and Table \ref{tab:rho-tau2} in   
Appendix \ref{sec:contraction},   
along with the following identities.
\begin{align}
&r_1^+(z)-r_1^-(z)=a_3(qz)-a_3(q^{-1}z)\,,
\label{r1r1-1}\\
&r_1^+(q^{k_1+2}z)-r_1^-(q^{-k_1-2}z)=a_2(q^{-1}z)-a_2(qz)\,,
\label{r1r1-2}\\
&t_1^+(z)-t_1^-(z)=a_1(qz)-a_1(q^{-1}z)\,,
\label{t1t1-1}\\
&t_1^+(q^{k_1+2}z)-t_1^-(q^{-k_1-2}z)=a_0(q^{-1}z)-a_0(qz)\,.
\label{t1t1-2}
\end{align}

As an example let us verify the anti-commutativity between 
$\sigma_2$ and $\rho_1$. 
We find from Table \ref{tab:art} 
\begin{align*}
(q-q^{-1})z\rho_1(z)\cdot\sigma_2(w)
&= 
\frac{1}{q^{-k_1-1}z-w}:e^{r_1^+(z)+a_2(w)}:
-\frac{1}{q^{k_1+1}z-w}:e^{r_1^-(z)+a_2(w)}:\,,
\\
\sigma_2(w)\cdot(q-q^{-1})z\rho_1(z) 
&=
\frac{1}{w-q^{-k_1-1}z}:e^{r_1^+(z)+a_2(w)}:
-\frac{1}{w-q^{k_1+1}z}:e^{r_1^-(z)+a_2(w)}:\,.
\end{align*}
Upon integration and shifting contours, we obtain
\begin{align*}
(q-q^{-1})[\rho_1,\sigma_2]_+
&= \int\frac{dz}{z}\Bigl(:e^{r_1^+(z)+a_2(q^{-k_1-1}z)}:
-:e^{r_1^-(z)+a_2(q^{k_1+1}z)}:\Bigr)\\
&= \int\frac{dz}{z}\Bigl(:e^{r_1^+(q^{k_1+2}z)+a_2(qz)}:
-:e^{r_1^-(q^{-k_1-2}z)+a_2(q^{-1}z)}:\Bigr)\,.
\end{align*}
Now it suffices to apply \eqref{r1r1-2} which implies 
$r_1^+(q^{k_1+2}z)+a_2(qz)=r_1^-(q^{-k_1-2}z)+a_2(q^{-1}z)$. 

The other cases can be verified similarly.
\end{proof}
\bigskip

In the following sections we shall explain the algebraic 
background of these formulas. 
We shall show that the deformed screening currents
are closely connected with generators of quantum toroidal algebras, 
see Proposition \ref{prop:rhoFF} and discussions in Section \ref{sec:D21}.

\bigskip

\section{Quantum Toroidal $\gl_2$}\label{sec:toroidal gl2}
In this section we introduce actions of the quantum toroidal $\mathfrak{gl}_2$
algebra $\E_2(q_1,q_2,q_3)$
on the total Fock space $\cH$. 
For that purpose, we start from the Wakimoto representation of the 
quantum affine algebra $U_q\widehat{\mathfrak{sl}}_2\subset\E_2(q_1,q_2,q_3)$. 
 
The Wakimoto representation is defined on the Fock space of 
three free fields \cite{Sh}. Re-expressing them in terms of 
$a_i(z)$, $i=1,2,3$, we are led to 
a realization of the Wakimoto representation on our Fock space $\cH$.
Leaving the technical details to Appendix \ref{sec:Wakimoto}, 
we present the result below. 

Let us introduce fields $u_1^\pm(z)$, $v_1^\pm(z)$,
\begin{align}
&u^\pm_1(z)=Q_{u_1}+u_{1,0}\log z
\mp a_{3,0}\log q-\sum_{n\neq0}\frac{u^\pm_{1,n}}{n}z^{-n}\,,
\label{u1n}\\
&v^\pm_1(z)=
-\bigl(Q_{u_1}+u_{1,0}\log z\bigr)
\pm a_{2,0}\log q -\sum_{n\neq0}\frac{v^\pm_{1,n}}{n}z^{-n}\,.
\label{v1n}
\end{align}
The oscillator part is given by
\begin{align}
&u^+_{1,n}=
\begin{cases}
\frac{[n]}{[(k_3+2)n]}\Bigl(q^{k_1n}a_{1,n}+q^{-(k_2+2)n}a_{2,n}\Bigr)&(n<0),\\
-\frac{[n]}{[(k_2+2)n]}\Bigl(a_{1,n}+q^{-(k_2+2)n}a_{3,n}\Bigr) & (n>0),\\
\end{cases}
\label{u1pn}\\
&u^-_{1,n}=u^+_{1,n}-(q^n-q^{-n})a_{3,n}\,,
\label{u1nn}\\
&v^-_{1,n}=
\begin{cases}
-\frac{[n]}{[(k_3+2)n]}\Bigl(a_{1,n}+q^{-(k_3+2)n}a_{2,n}\Bigr)&(n<0),\\
\frac{[n]}{[(k_2+2)n]} \Bigl(q^{k_1n}a_{1,n}+q^{-(k_3+2)n}a_{3,n}\Bigr) & (n>0),\\
\end{cases}
\label{v1nn}\\
&v^+_{1,n}=v^-_{1,n}-(q^n-q^{-n})a_{2,n}\,.
\label{v1pn}
\end{align}
The zero mode part is 
\begin{align}
&Q_{u_1}=Q^0_{u_1} +s \bar{Q}_{u_1}\,,
\quad u_{1,0}=u^0_{1,0} +\frac{s}{1+k_1s}\bar{u}_{1,0}\,,
\label{Qu1s}
\end{align}
with
\begin{align}
& Q^0_{u_1}=\frac{1}{k_3+2}\bigl(Q_{a_1}+Q_{a_2}\bigr)\,,
\quad
u^0_{1,0}=-\frac{1}{k_2+2}\bigl(a_{1,0}+a_{3,0}\bigr)\,,
\label{Q0u0}\\
&\bar{Q}_{u_1}=\frac{1}{k_3+2}\bigl(k_1Q_{a_1}-(k_2+2)Q_{a_2}-(k_3+2)Q_{a_3}
\bigr)\,,
\label{barQ}\\
&\bar{u}_{1,0}=\frac{1}{k_2+2}\bigl(
k_1 a_{1,0}-(k_2+2)a_{2,0}-(k_3+2)a_{3,0}\bigr)\,.
\label{baru}
\end{align}
Here  $s$ is an arbitrary parameter. We will make a specific choice
later on, see \eqref{choose-s}.

Set further
\begin{align}
h_{1,0}&=-\frac{1}{1+k_1s}\bar{u}_{1,0}\,,
\label{H10}\\
h_{1,n}&=\frac{1}{n}
\Bigl([k_1n]a_{1,n}-[(k_2+2)n]a_{2,n}-[(k_3+2)n]a_{3,n}\Bigr)
\times
\begin{cases}
-\frac{[n]}{[(k_2+2)n]} & (n>0),\\
\frac{[n]}{[(k_3+2)n]} & (n<0).\\
\end{cases}
\label{h1n}
\end{align}

Let $E_1(z),F_1(z),K_1^{\pm1}, H_{1,n}$ ($n\neq 0$)
be the standard generators of $U_q\widehat{\mathfrak{sl}}_2$
(for our convention see Appendix \ref{sec:toridal}).
\begin{prop}\label{prop:sl2}
The following assignment gives the Wakimoto representation of  $U_q\widehat{\mathfrak{sl}}_2$
of level $q^{k_1}$:
\begin{align}
&-(q-q^{-1})E_1(z)\mapsto :e^{u_1^+(z)}:-:e^{u_1^-(z)}: \,,
\quad
(q-q^{-1})F_1(z)\mapsto :e^{v_1^+(z)}:-:e^{v_1^-(z)}: \,,
\label{sl2-EF}\\
&K_1\mapsto q^{h_{1,0}}\,,
\quad H_{1,n}\mapsto h_{1,n}\,.
\label{sl2-H}
\end{align}
\end{prop}
On the Fock space $\cH$ we have an extra boson consisting of $\{a_{0,n}\}$. 
The combination 
\begin{align*}
Z_n=(q^n+q^{-n})a_{0,n}+\sum_{i=1}^3(q^{(k_i+1)n}+q^{-(k_i+1)n})a_{i,n}
\end{align*}
is a Heisenberg algebra commuting with $\{a_{i,n}\}$, $i=1,2,3$, and 
hence with the action of $U_q\widehat{\mathfrak{sl}}_2$. 
Altogether we have a representation of $U_q\widehat{\mathfrak{gl}}_2$ on $\cH$. 
This representation can be further lifted to one of
 $\E_2(q_1,q_2,q_3)$ via the following evaluation morphism: 
\begin{prop}\label{prop:eval}\cite{Mi} 
Let $\bar\E_2(q_1,q_2,q_3)$ denote the quotient of $\E_2(q_1,q_2,q_3)$ 
by the relation $C=q_3$, where $C$ denotes the central element.
Then there exists a homomorphism of algebras
$\bar\E_2(q_1,q_2,q_3)\to  U_q\widehat{\mathfrak{gl}}_2$.
\end{prop}

We use the formulas in \cite{FJM1} for the evaluation homomorphism 
to arrive at the following conclusion. 
\begin{thm}\label{thm:toroidal-rep}
Let $\E_{2,1}$ denote 
the quantum toroidal algebra $\E_2(q_1,q_2,q_3)$ with parameters
\[
 q_1=q^{-k_1-2}\,,\quad q_2=q^2\,,\quad q_3=q^{k_1}\,.
\]
Then the following assignment along with \eqref{sl2-EF}, \eqref{sl2-H}
gives a representation $\pi_1:\E_{2,1}\to \End\cH$.  
\begin{align}
&-(q-q^{-1})E_0(z)\mapsto :e^{\uh_1^+(z)}:-:e^{\uh_1^-(z)}: \,,
\quad
(q-q^{-1})F_0(z)\mapsto :e^{\vh_1^+(z)}:-:e^{\vh_1^-(z)}: \,,
\label{sl2-EF0}
\\
&
K_0\mapsto q^{-h_{1,0}}\,,
\quad H_{0,n}\mapsto \hh_{1,n}\,,
\label{sl2-H0}
\end{align}
where
\begin{align}
&\uh^\pm_1(z)=
-\bigl(Q_{u_1}+u_{1,0}\log z\bigr)
\mp a_{2,0}\log q-\sum_{n\neq0}\frac{\uh^\pm_{1,n}}{n}z^{-n}\,,
\label{uh1}\\
&\vh^\pm_1(z)=Q_{u_1}+u_{1,0}\log z
\pm a_{3,0}\log q
-\sum_{n\neq0}\frac{\vh^\pm_{1,n}}{n}z^{-n}\,,
\label{vh1}\\
&\hh_{1,n}=\frac{1}{n}
\Bigl([k_1n]a_{0,n}-[(k_3+2)n]a_{2,n}-[(k_2+2)n]a_{3,n}\Bigr)
\times
\begin{cases}
-\frac{[n]}{[(k_2+2)n]} & (n>0),\\
\frac{[n]}{[(k_3+2)n]} & (n<0),\\
\end{cases}
\label{hh1}
\end{align}
and 
\begin{align}
&\uh^+_{1,n}=
\begin{cases}
\frac{[n]}{[(k_3+2)n]} \Bigl(q^{k_1n}a_{0,n}+q^{-(k_2+2)n}a_{3,n}\Bigr)
& (n<0)\\
-\frac{[n]}{[(k_2+2)n]} \Bigl(a_{0,n}+q^{-(k_2+2)n}a_{2,n}\Bigr) & (n>0)\\
\end{cases}
\label{uh1pn}\\
&\uh^-_{1,n}=\uh^+_{1,n}-(q^n-q^{-n})a_{2,n}\,,
\label{uh1nn}\\
&\vh^-_{1,n}=
\begin{cases}
-\frac{[n]}{[(k_3+2)n]}\Bigl(a_{0,n}+q^{-(k_3+2)n}a_{3,n}\Bigr)
& (n<0)\\
\frac{[n]}{[(k_2+2)n]} \Bigl(q^{k_1n}a_{0,n}+q^{-(k_3+2)n}a_{2,n}\Bigr) 
& (n>0)\\
\end{cases}
\label{vh1nn}\\
&\vh^+_{1,n}=\vh^-_{1,n}-(q^n-q^{-n})a_{3,n}\,.
\label{vh1pn}
\end{align}
\end{thm}
\begin{lem}
The bosonic screening $\rho_1$ and the 
fermionic screenings $\sigma_2,\sigma_3$
(anti-)commute with the action of $\E_{2,1}$.  
\end{lem}
\begin{proof}
This can be verified in the same way as in Proposition \ref{prop:q-screening}. 
We use Table \ref{tab:UVsigma} and the identities 
\begin{align*}
&u_1^+(z)-u_1^-(z)=-a_3(q z)+a_3(q^{-1} z)\,,
\\
&v_1^+(z)-v_1^-(z)=-a_2(q^{-1} z)+a_2(q z)\,,
\\
&\hat u_1^+(z)-\hat u_1^-(z)=-a_2(q z)+a_2(q^{-1} z)\,,
\\
&\hat v_1^+(z)-\hat v_1^-(z)=-a_3(q^{-1} z)+a_3(q z)\,.
\end{align*}
\end{proof}
\medskip

Cyclically permuting the indices $(1,2,3)$, 
we obtain three different
actions of quantum toroidal $\gl_2$ algebras on $\cH$.
\begin{cor}
Let $\E_{2,i}=\E_2(q^{-k_i-2},q^2,q^{k_i})$, $i=1,2,3$.
Then there exist representations $\pi_i:\E_{2,i}\to \End\cH$. 
\end{cor}
We are not aware of any reasonable relations between algebras $\E_{2,i}$
with different $i$.
These actions do not mutually commute, in fact 
even the quantum affine subalgebras do not. 
The situation is different from the quantum toroidal 
$(\mathfrak{gl}_m,\mathfrak{gl}_n)$ duality \cite{FJM2}
 where the quantum toroidal 
algebras do not commute but 
the quantum affine subalgebras mutually commute. 

\medskip

\section{Integrals of Motion}\label{sec:IM}
In \cite{FJM} it was shown that 
the standard transfer matrix construction provides 
a commutative subalgebra inside a completion of $\E_2$. 
By ``transfer matrix'' we mean a weighted trace of the universal 
R matrix of $\E_2$. In the course of taking trace, one is naturally led to
consider dressed currents
\begin{align*}
F_i^{dress}(z,p)&=
F_i(z)\prod_{\ell\ge 0}\bar{K}_i^+\bigl(p^\ell z\bigr)^{-1}
\\
&=F_i(z)\exp\Bigl(-(q-q^{-1})\sum_{n>0}\frac{1}{1-p^n}H_{i,n}z^{-n}\Bigr)\,.
\end{align*}
The Taylor coefficients of the transfer matrix are then
multiple integrals involving products of the dressed currents, giving 
the deformed analog of the ``non-local integrals of motion''. 
The aim of this section is to 
discuss integrals of motion 

We begin by observing that 
the bosonic screening currents $\rho_i(z),\tau_i(z)$ 
are the images of these dressed currents on the Fock space $\cH$, 
modulo a simple modification in the zero mode.  
In the following we set 
\begin{align*}
p_i=q^{2(k_i+2)}\,,\quad i=1,2,3\,,
\end{align*}
and assume that $|p_i|<1$ for $i=1,2,3$. 
This implies in particular that $|q^4|=|p_1p_2p_3|<1$.

\begin{prop}\label{prop:rhoFF}
We have the relations
\begin{align*}
&z\rho_3(z) =e^{ s \bar{Q}_{u_1}}
\pi_1\Bigl(F_1^{dress}(z,p_3)\Bigr)z^{-\lambda h_{1,0}}
\,,
\\
&z\tau_3(z)= e^{-s \bar{Q}_{u_1}}
\pi_1\Bigl(F_0^{dress}(z,p_3)\Bigr)z^{\lambda h_{1,0}}
\,,
\end{align*}
where $\bar{Q}_{u_1}$ is given in \eqref{barQ} and
\begin{align*}
\lambda=\frac{1-(k_2+2)s}{k_3+2}\,.
\end{align*}
\end{prop}
\begin{proof}
This follows from the identities
\begin{align*}
&r_3^\pm(z)-v^\pm_1(z)=s \bar{Q}_{u_1}
-\lambda h_{1,0}\log z
-(q-q^{-1})\sum_{n>0}\frac{1}{1-p_3^n}h_{1,n}z^{-n}\,, 
\\
&t_3^\pm(z)-\vh^\pm_1(z)=
-s\bar{Q}_{u_1}+\lambda h_{1,0}\log z
-(q-q^{-1})\sum_{n>0}\frac{1}{1-p_3^n}\hh_{1,n}z^{-n}\,,
\end{align*}
which can be checked directly using
\eqref{u1n}--\eqref{h1n}, \eqref{uh1}--\eqref{vh1pn} and
\eqref{Qu1s}. 
\end{proof}
\medskip

Proposition \ref{prop:rhoFF} allows us to rewrite 
\begin{align*}
&\prod_{1\le j\le N}^{\curvearrowright} 
\pi_1\bigl(F_0^{dress}(x_{0,j},p_3)\bigr)
\prod_{1\le j\le N}^{\curvearrowright}
\pi_1\bigl(F_1^{dress}(x_{1,j},p_3)\bigr)\\
&=
\prod_{1\le j\le N}^{\curvearrowright} x_{0,j}\tau_{3}(x_{0,j})
\prod_{1\le j\le N}^{\curvearrowright}x_{1,j}\rho_{3}(x_{1,j})
 \times 
\prod_{j=1}^Nx_{0,j}^{\frac{2j}{k_3+2}}x_{1,j}^{-\frac{2(N-j)}{k_3+2}}
\prod_{j=1}^N\Bigl(\frac{x_{1,j}}{x_{0,j}}\Bigr)^{\lambda h_{1,0}}\,,
\end{align*}
where $\prod_{1\le j\le N}^{\curvearrowright} A_j=A_1A_2\cdots A_N$.
Inserting this into the known expressions for the integrals of motion
associated with the quantum toroidal $\mathfrak{gl}_2$ algebra \cite{FJM}
we obtain the following formula
\footnote{We use \cite{FJM}, eq.(3.11),   
choosing $\bar p_1=1$ in the argument
of $\vartheta_\nu$. Also the factor $q^{-H_{1,0}/2}$ there 
should be corrected to $q^{H_{1,0}/2}$.}
\begin{align}
&\mathbb{G}^{(\nu)}_{1,N}= 
\int\!\!\cdots \!\!\int
\prod_{1\le j\le N\atop t=0,1}\frac{d x_{t,j}}{x_{t,j}}
\prod_{1\le j\le N}^{\curvearrowright} x_{0,j}\tau_{3}(x_{0,j})
\prod_{1\le j\le N}^{\curvearrowright}x_{1,j}\rho_{3}(x_{1,j})
\cdot K^{(\nu)}_{1,N}\bigl(\mathbf{x};p_{3}\bigr)\,,
\label{G1}
\end{align}
where $\mathbf{x}=(x_{0,1},\ldots,x_{0,N},x_{1,1},\ldots,x_{1,N})$, 
$\nu=0,1$, 
\begin{align}
&K^{(\nu)}_{1,N}\bigl(\mathbf{x};p\bigr)
=\frac{\prod_{t=0,1}\prod_{j<l}
\Theta_{p}(x_{t,l}/x_{t,j},q^2x_{t,l}/x_{t,j})}
{\prod_{j,l}
\Theta_{p}(q^{k_1+2}x_{1,l}/x_{0,j},q^{-k_1}x_{1,l}/x_{0,j})
}
\vartheta_\nu\Bigl(
q^{h_{1,0}/2}\prod_{j=1}^N\frac{x_{1,j}}{x_{0,j}},p\Bigr)
\label{kernel}\\
&\quad \times 
\prod_{j=1}^N 
x_{0,j}^{1-2j+\frac{2j}{k_{3}+2}}
x_{1,j}^{1-2(N-j)-\frac{2(N-j)}{k_{3}+2}}
\Bigl(\prod_{j=1}^N\frac{x_{1,j}}{x_{0,j}}\Bigr)^{\lambda h_{1,0}}
\,,\nn
\end{align}
and $\Theta_p(z)=\prod_{n=1}^\infty(1-p^{n-1}z)(1-p^nz^{-1})(1-p^n)$,
$\vartheta_\nu(z,p)=\sum_{n\in\Z+\nu/2}p^{n^2}z^{2n}$.
The contours of integration are deformations of the unit circle 
$|x_{0,j}|=|x_{1,j}|=1$ from the domain where the parameters satisfy
$|q^{-k_1-2}|<1,|q^{k_1}|<1$. 
More specifically, for each $i=0,1$ and $j=1,\ldots,N$, the $x_{i,j}$ 
integration contour is such that the series of poles
\begin{align*}
&p_3^mq^{-k_1-2}x_{1-i,l}\,,\ p_3^mq^{k_1}x_{1-i,l}\,
\quad (1\le l\le N, m\ge0) 
\end{align*}
are inside and 
\begin{align*}
&p_3^{-m}q^{k_1+2}x_{1-i,l}\,,\ p_3^{-m}q^{-k_1}x_{1-i,l}\,
\quad (1\le l\le N, m\ge0) 
\end{align*}
are outside.

From now on, we choose 
\begin{align}
 s=\frac{1}{2(k_2+2)}
\label{choose-s}
\end{align}
so that $\lambda=\frac{1}{2(k_3+2)}$. 
The point of the choice \eqref{choose-s} is to make
the kernel function \eqref{kernel} quasi-periodic, 
having a factor of quasi-periodicity independent of $h_{1,0}$.
That is, 

\begin{lem}\label{lem:quasi}
We have
\begin{align*}
&K^{(\nu)}_{1,N}(\mathbf{x};p)\bigl|_{x_{0,j}\to p x_{0,j}}
=q^{2}K^{(\nu)}_{1,N}(\mathbf{x};p)
\,,
\quad
K^{(\nu)}_{1,N}(\mathbf{x};p)\bigl|_{x_{1,j}\to p x_{1,j}}
=q^{2}K^{(\nu)}_{1,N}(\mathbf{x};p)\,.
\end{align*}
\end{lem}
\begin{proof}
This can be checked directly using  
$\Theta_p(pz)=-z^{-1}\Theta_p(z)$, 
$\vartheta_\nu(pz;p)=p^{-1}z^{-2}\vartheta_\nu(z;p)$.
\end{proof}

With the choice \eqref{choose-s}, 
and after normal ordering, the integrand of \eqref{G1}  comprises only
integer powers except for the factor $(\prod_{j=1}^N x_{1,j}/x_{0,1})^{u_{1,0}}$, 
where $u_{1,0}$ given in \eqref{Qu1s} simplifies to
\begin{align*}
&u_{1,0}=\frac{1}{k_3-k_2}\bigl(2a_{1,0}+a_{2,0}+a_{3,0}\bigr)\,.
\end{align*}
Hence the action of $U_q\widehat{\mathfrak{sl}}_2\subset \E_{2,1}$ and $\mathbb{G}^{(\nu)}_{1,N}$
are defined on sectors on which 
$2a_{1,0}+a_{2,0}+a_{3,0}$ has eigenvalues in $(k_3-k_2)\Z$.

Cyclically permuting the indices $(1,2,3)$ we obtain three kinds of 
integrals of motion $\mathbb{G}^{(\nu)}_{i,N}$ associated with
$\E_{2,i}$, $i=1,2,3$.  
These operators are well-defined 
on sectors $\cH_{N_1,N_2,N_3}$ ($N_i\in\Z$) 
on which  $a_{i,0}$ have eigenvalues 
\begin{align*}
a_{1,0}=\frac{3}{4}(k_2-k_3)N_1-\frac{1}{4}(k_3-k_1)N_2
-\frac{1}{4}(k_1-k_2)N_3
\quad \&\ cycl..
\end{align*}
\medskip

Even though the algebras $\E_{2,i}$ do not commute with each other, 
the corresponding integrals of motion do. The following ``triality'' is 
the main result of this paper. 
\begin{thm}\label{thm:triality}
On sectors $\cH_{N_1,N_2,N_3}$, 
the three kinds of integrals of motion 
are mutually commutative. Namely we have
\begin{align*}
&[\mathbb{G}^{(\mu)}_{i,M},\mathbb{G}^{(\nu)}_{j,N}]=0\quad 
(i,j=1,2,3\,, M,N\ge1\,, \mu,\nu=0,1)\,.
\end{align*}
\end{thm}
\begin{proof}
When $i=j$ the commutativity is known (and is evident by construction). 
So it suffices to consider $i=1,j=2$. 
The product 
$\mathbb{G}^{(\mu)}_{1,M}\mathbb{G}^{(\nu)}_{2,N}$
is an integral 
\begin{align*}
&\int\!\cdots\!\int\prod_{1\le j\le M\atop t=0,1}
\frac{dx_{t,j}}{x_{t,j}}\prod_{1\le l\le M\atop t=0,1}
\frac{dy_{t,l}}{y_{t,l}}\, K^{(\mu)}_{1,M}(\mathbf{x};p_3)
 K^{(\nu)}_{2,N}(\mathbf{y};p_1)
\\
&\times
\prod_{1\le j\le M}^{\curvearrowright} x_{0,j}\tau_{3}(x_{0,j})
\prod_{1\le j\le M}^{\curvearrowright}x_{1,j}\rho_{3}(x_{1,j})
\prod_{1\le l\le N}^{\curvearrowright} y_{0,l}\tau_{1}(y_{0,l})
\prod_{1\le l\le N}^{\curvearrowright}y_{1,l}\rho_{1}(y_{1,l})\,.
\end{align*}
From Table \ref{tab:rho-tau2}, contractions of 
$\tau_3(x_{0,j})$ or $\rho_3(x_{1,j})$ with
$\tau_1(y_{0,l})$ or $\rho_1(y_{1,l})$ are rational functions 
with simple poles. 
The commutator $[\mathbb{G}^{(\mu)}_{1,M},\mathbb{G}^{(\nu)}_{2,N}]$
is a sum of terms obtained by taking residues from each such contraction.

Let us use $z$, $w$ to represent variables that enter the relevant 
contraction. There are four cases to consider,
\begin{align*}
&\rho^\pm_3(z) \rho^\pm_1(w)
=q^{\mp 1}\frac{1-q^{\mp k_1}w/z}{1-q^{\mp(k_1+2)}w/z}:e^{r^\pm_3(z)+r^\pm_1(w)}:+\cdots\,,
\\
&\rho^\pm_3(z) \tau^\pm_1(w)
=q^{\pm 1}\frac{1-q^{\pm k_3}w/z}{1-q^{\pm(k_3+2)}w/z}
:e^{r^\pm_3(z)+t^\pm_1(w)}:+\cdots\,,
\\
&\tau^\mp_3(z) \rho^\pm_1(w)
=q^{\pm 1}\frac{1-q^{\mp 2}w/z}{1-w/z}:e^{t^\mp_3(z)+r^\pm_1(w)}:+\cdots\,,
\\
&\tau^\mp_3(z) \tau^\pm_1(w)
=q^{\mp 1}\frac{1-q^{\pm (k_2+2)}w/z}{1-q^{\pm k_2}w/z}
:e^{t^\pm_3(z)+t^\pm_1(w)}:+\cdots\,.
\end{align*}
In the first three cases, integration with respect to $z$
yields the same term with opposite signs, 
due to the identities which follow from \eqref{r1r1-1}--\eqref{t1t1-2}:
\begin{align*}
&r^+_3(w)-r^-_3(w)=-r^+_1(q^{k_1+2}w)+r_1^-(q^{-k_1-2}w)\,, 
\\
&r^+_3(q^{k_3+2}w)-r_3^-(q^{-k_3-2}w)=-t_1^+(w)+t_1^-(w)\,, 
\\
&t^+_3(w)-t^-_3(w)=r^+_1(w)-r^-_1(w)\,.
\end{align*}
Therefore the residues sum up to zero. 
In the last case we need a shift of arguments. 
Noting that $p_1q^{k_2}=p_3^{-1}q^{-k_2}$, we have
\begin{align*}
&\bigl(t^-_3(q^{k_2}w)+t^+_1(w)\bigr)\bigl|_{w\to p_1w}
-\bigl(t^+_3(q^{-k_2}w)+t^-_1(w)\bigr)\\
&=-\bigl(t^+_3(q^{-k_2}w)-t^-_3(p_3^{-1}q^{-k_2}w)\bigr)
+\bigl(t^+_1(p_1 w)-t^-_1(w)\bigr)=0\,.
\end{align*}
Thanks to Lemma \ref{lem:quasi},
the product of the kernel functions remains unchanged
under this shift,  
\begin{align*}
K^{(\mu)}_{3,M}(\mathbf{x};p_3)\bigl|_{x_{j,0}\to p_3^{-1}x_{j,0}}
K^{(\nu)}_{1,N}(\mathbf{y};p_1)\bigl|_{y_{l,0}\to p_1 y_{l,0}}
=K^{(\mu)}_{3,M}(\mathbf{x};p_3)K^{(\nu)}_{1,N}(\mathbf{y};p_1)\,.
\end{align*}
Case by case analysis shows that the contours also 
match after shift.  
\end{proof}

\section{$D(2,1;\alpha)$ structure}\label{sec:D21}
Finally we make some comments regarding the deformed
fermionic screening currents and the possible connection with 
the ``quantum toroidal $D(2,1;\alpha)$'' algebra.

Let us begin with the affine superalgebra $\hat{D}(2,1;\alpha)$ with
$\alpha={k_1}/(k_3+2)$. 
We choose all simple roots $\alpha_i$ to be fermionic, that is 
$(\alpha_i,\alpha_i)=0$,
so that the Dynkin diagram is as in figure \ref{fig1} below.
The quantum affine superalgebra
$U_\ssq\hat{D}(2,1;\alpha)$ has been studied in \cite{HSTY}, where 
the Drinfeld realization is obtained. 
Here we take the deformation parameter as $\ssq=q^{-k_3-2}$. 

We are interested only in the Borel subalgebra of 
$U_\ssq\hat{D}(2,1;\alpha)$. The latter is generated by 
commutative elements $\ssh_{i,n}$ and fermionic currents $\xi_i(z)$ 
($i=1,2,3$, $n>0$), satisfying
\begin{align}
&[\ssh_{i,n},\xi_j(z)]=\frac{[(\alpha_i,\alpha_j)n]}{n}z^n\, \xi_j(z)
\,,\label{D21-1}\\
&(z-q^{(\alpha_i,\alpha_j)}w)\xi_i(z)\xi_j(w)
=(w-q^{(\alpha_i,\alpha_j)}z)\xi_j(w)\xi_i(z)\,,
\label{D21-2}\\
&[(\alpha_1,\alpha_3)][\![ [\![  \xi_1(z_1),\xi_2(z_2) ]\!], \xi_3(z_3) ]\!]
=[(\alpha_1,\alpha_2)][\![ [\![ \xi_1(z_1),\xi_3(z_3) ]\!], \xi_2(z_2) ]\!]
\quad \& \ cycl.\,.
\label{D21-3}
\end{align}
Here, for elements $X$, $Y$
of parity ${\bar X}, {\bar Y}\in\{0,1\}$ 
and weight $\wt X, \wt Y$ we set
\begin{align*}
[\![X,Y]\!]=XY-(-)^{\bar X \bar Y}q^{-(\wt X,\wt Y)}YX \,.
\end{align*}

To our knowledge, quantum toroidal version of the $D(2,1;\alpha)$ 
algebra has not been defined in the literature. 
Guided by the usual construction of quantum toroidal algebras, 
we consider an algebra $\textsf{B}$ generated by 
symbols  $\ssh_{i,n}$, $\xi_i(z)$ for $i=0,1,2,3$
satisfying \eqref{D21-1}--\eqref{D21-3}
for all $i,j,k\in\{0,1,2,3\}$.
We take $\textsf{B}$ 
as a working definition of the Borel part of the quantum toroidal
$D(2,1;\alpha)$ algebra.
\bigskip

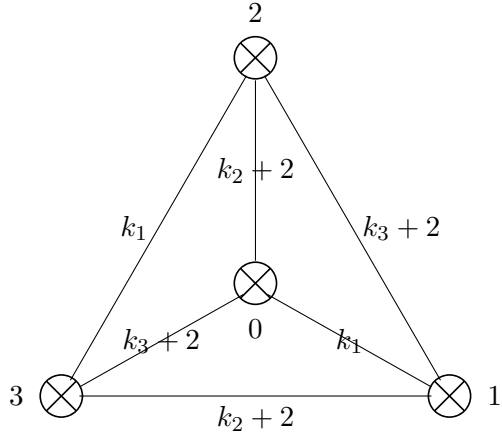
\begin{figure}[H]
\begin{center}
\begin{tikzpicture}[scale=0.6]
\coordinate  (A) at  (0,5);
\coordinate  (B) at (0,0);
\coordinate   (C) at (4.3,-2.5);
\coordinate  (D) at  (-4.3,-2.5);
\draw (0,4.5) --  node {$k_2+2$} (0,0.5);
\draw (0.2,4.6) -- node[right]{$k_3+2$} (4.1,-2.15);
\draw (-0.2,4.6) -- node[left]  {$k_1$}(-4.1,-2.15);
\draw (0.27,-0.25) -- node  {$k_1$} (3.9,-2.3);
\draw (-0.27,-0.25) -- node {$k_3+2$}  (-3.9,-2.3);
\draw  (3.9,-2.5) --node[below] {$k_2+2$} (-3.9,-2.5);
\node at (A) {\Huge $\otimes$};
\node at (B) {\Huge$\otimes$};
\node at (C)   {\Huge$\otimes$};
\node  at (D)  {\Huge$\otimes$};
\node at (0,6) {2};
\node at (0,-1) {0};
\node at  (-5.3,-2.5) {3};
\node at  (5.3,-2.5) {1};
\end{tikzpicture}
\caption{Dynkin diagram of $\hat{D}(2,1;\alpha)$, $\alpha=\frac{k_1}{k_3+2}$
The numbers inscribed on the edges indicate the scalar
product of roots, e.g., $(\alpha_0,\alpha_1)=k_1$. 
 \label{fig1}}
\end{center}
\end{figure}
\bigskip

It is not hard to find a free field realization of $\textsf{B}$. 
Let 
\begin{align*}
\varepsilon_i=\begin{cases}
1 & i=0,1\,,\\
-1 & i=2,3 \,, \\
\end{cases}
\end{align*}
and define matrices  $A,B$ by 
\begin{align*}
&A_{i,j}=\varepsilon_i\varepsilon_j 
q^{-(\alpha_i,\alpha_j)\varepsilon_i\varepsilon_jn}\,, \quad
B_{i,j}=q^{(\alpha_i,\alpha_j)n}-q^{-(\alpha_i,\alpha_j)n}\,,
\quad i,j\in\{0,1,2,3\}.
\end{align*}
Introducing one pair of zero mode $Q,a$ with $[a,Q]=1$, 
we consider fields 
\begin{align*}
\ssx_i(z)=\varepsilon_i \bigl(Q+a\log z\bigr) -
\sum_{n\neq 0}\frac{\ssx_{i,n}}{n}z^{-n} \,, 
\quad i=0,1,2,3\,,
\end{align*}
where the oscillators $\ssx_{i,n}$ satisfy
\begin{align*}
[\ssx_{i,n},\ssx_{j,-n}]=n A_{i,j}\,.
\end{align*}
Then the following assignment 
gives a realization of \eqref{D21-1}--\eqref{D21-3} with $i=0,1,2,3$. 
\begin{align*}
&\xi_i(z)\mapsto :e^{\ssx_i(z)}:\,,
\\
&(q-q^{-1})n\ssh_{i,n}\mapsto \sum_{j=0}^3(BA^{-1})_{i,j}\ssx_{j,n}\,.  
\end{align*}

The fermionic fields $a_i(z)$ considered in the previous sections can be 
viewed as currents dressed from $\ssx_i(z)$ by $\ssh_{i,n}$. 
Indeed, setting 
\begin{align*}
\ssa_{i,n}&=\ssx_{i,n}-(q-q^{-1})\frac{q^{2n}}{1-q^{2n}}n\ssh_{i,n}\,, 
\\
\ssa_{i,-n}&=\ssx_{i,-n}\,,
\end{align*}
for $n>0$, we obtain
\begin{align*}
[\ssa_{i,n},\ssa_{j,-n}] =n
\frac{[((\alpha_i,\alpha_j)+\varepsilon_i\varepsilon_j)n]}{[n]}\,.
\end{align*}
These are exactly the relations \eqref{aii},\eqref{a12},\eqref{a01}
for the oscillator part of $a_i(z)$.

The quantum toroidal $D(2,1;\alpha)$ algebra 
and its commutative subalgebra
warrant further investigation. 
For that purpose the shuffle algebra approach might be useful. 
We hope to return to this subject in the future.

\bigskip

\appendix
\section{Quantum toroidal $\gl_2$ algebra}\label{sec:toridal}
We give here the defining relations of quantum toroidal $\gl_2$ 
algebra $\E_2=\E_2(q_1,q_2,q_3)$. 

Let $(a_{i,j})=
\begin{pmatrix}2 &-2\\-2&2\\ \end{pmatrix}$ be the Cartan matrix of type $A^{(1)}_1$.
Define $q,d$ by $q_2=q^2$, $q_1=d/q$ and $q_3=1/(dq)$. 
For $r\neq 0$ we set
\begin{align*}
&a_{i,j}(r)=\frac{[r]}{r} \times 
\begin{cases}
	      q^r+q^{-r} & (i\equiv j \bmod 2),\\
              -d^r-d^{-r}& (i\not\equiv j \bmod2).
\end{cases}
\end{align*}
We set further
\begin{align*}
g_{i,j}(z,w)=\begin{cases}
	      z-q_2w & (i\equiv j \bmod 2),\\
              (z-q_1w)(z-q_3w)& (i\not\equiv j \bmod2).
	     \end{cases}
\end{align*}
Algebra $\E_2(q_1,q_2,q_3)$ is generated by elements
$E_{i,k},F_{i,k},H_{i,r}$ 
and invertible elements $K_i$, $C$, where
$i=0,1$, $k\in\Z$, $r\in\Z\backslash\{0\}$. 
We present below the defining relations 
in terms of generating series
\begin{align*}
&E_i(z)=\sum_{k\in\Z}E_{i,k}z^{-k}, 
\quad
F_i(z)=\sum_{k\in\Z}F_{i,k}z^{-k},\\
&K^{\pm}_i(z)=K_i^{\pm1}\bar{K}^{\pm}_i(z)\,,
\quad
\bar{K}^{\pm}_i(z)=
\exp\Bigl(\pm(q-q^{-1})\sum_{r>0}H_{i,\pm r}z^{\mp r}\Bigr)\,.
\end{align*}
The relations are as follows.

\bigskip

\noindent{\bf $C,K$ relations}
\begin{align*}
&\text{$C$ is central},\quad K_0K_1=1\,.
\\
&K_i E_j(z) K_i^{-1}=q^{a_{i,j}}E_j(z)\,,
\quad
K_i F_j(z) K_i^{-1}=q^{-a_{i,j}}F_j(z)\,.
\end{align*}
\bigskip

\noindent{\bf $H$-$H$, $H$-$E$ and $H$-$F$ relations}\quad
For $r\neq 0$, 
\begin{align*}
&[H_{i,r},E_j(z)]= a_{i,j}(r)C^{-(r+|r|)/2} \,z^r E_j(z)\,,\\
&[H_{i,r},F_j(z)]=-a_{i,j}(r)C^{-(r-|r|)/2} \,z^r F_j(z)\,,\\
&[H_{i,r},H_{j,s}]=\delta_{r+s,0} a_{i,j}(r)\frac{C^r-C^{-r}}{q-q^{-1}} \,.
\end{align*}
\bigskip

\noindent{\bf $E$-$F$ relations}
\begin{align*}
&[E_i(z),F_j(w)]=\frac{\delta_{i,j}}{q-q^{-1}}
(\delta\bigl(C\frac{w}{z}\bigr)K_i^+(w)
-\delta\bigl(C\frac{z}{w}\bigr)K_i^-(z))\,.
\end{align*}
\bigskip

\noindent{\bf $E$-$E$ and $F$-$F$ relations}
\begin{align*}
&(-1)^{i+j}g_{i,j}(z,w)E_i(z)E_j(w)+g_{j,i}(w,z)E_j(w)E_i(z)=0, 
\\
&(-1)^{i+j}g_{j,i}(w,z)F_i(z)F_j(w)+g_{i,j}(z,w)F_j(w)F_i(z)=0.
\end{align*}
\bigskip

\noindent{\bf Serre relations}
For $i\neq j$,
\begin{align*}
&\Sym_{z_1,z_2,z_3}[E_i(z_1),[E_i(z_2),[E_i(z_3),E_j(w)]_{q^2}]]_{q^{-2}} =0\,,
\\
&\Sym_{z_1,z_2,z_3}[F_i(z_1),[F_i(z_2),[F_i(z_3),F_j(w)]_{q^2}]]_{q^{-2}} =0\,.
\end{align*}
In the above we set $[A,B]_p=AB-pBA$ and 
\begin{align*}
&\Sym\ f(x_1,\dots,x_N) =\frac{1}{N!}
\sum_{\pi\in\GS_N} f(x_{\pi(1)},\dots,x_{\pi{(N)}})\,.
\end{align*}
\bigskip

We identify the quantum affine algebra  $U_q\widehat{\mathfrak{sl}_2}$ 
with the subalgebra of $\E_2(q_1,q_2,q_3)$ generated by 
$E_1(z),F_1(z), H_{1,r}$ and $K_1^{\pm1},C^{\pm1}$. 

\section{Comparison with Wakimoto representation}\label{sec:Wakimoto}
In this section we explain the connection between 
the representation of $\E_2$ in section \ref{sec:toroidal gl2}
and the Wakimoto representation of $U_q\widehat{\mathfrak{sl}_2}$. 

In \cite{Sh},  the Wakimoto representation of level $q^k$
is written 
in terms of the generators
\begin{align*}
&zJ^+(z) =E_1(q^{-k-2}z)\,,
\quad 
zJ^-(z) =F_1(q^{-k-2}z)\,,
\\
&J^3_n=q^{kn/2}H_{1,n}\,,
\quad \psi(z)=K^+_1(q^{-k/2}z)\,,
\quad \varphi(z)=K^-_1(q^{-k/2}z)\,,
\end{align*}
using three free fields $a(z),b(z),c(z)$. 
Non-trivial commutation relations among the 
Fourier modes of the latter are
\begin{align}
&[a_0,Q_a]=2(k+2)\,,\quad -[b_0,Q_b]=[c_0,Q_c]=4\,,
\label{abc1}\\
&[a_n,a_{-n}]=\frac{[2n][(k+2)n]}{n}\,,\quad
-[b_n,b_{-n}]=[c_n,c_{-n}]=\frac{[2n]^2}{n}\,.
\label{abc2}
\end{align}
For $x=a,b,c$ we set 
\begin{align*}
&x(L;M,N|z;\alpha)=\frac{L}{MN}Q_x+\frac{L}{MN}x_0\log z
-\sum_{n\neq0}\frac{[Ln]}{[Mn][Nn]}x_nz^{-n}q^{\alpha|n|} \,,
\\
&x(N|z;\alpha)=x(L;L,N|z;\alpha)\,.
\end{align*}

\begin{thm}\label{thm:Sh}\cite{Sh}
The following formula gives the Wakimoto representation of 
$U_q\widehat{\mathfrak{sl}_2}$ of level $q^k$.
\begin{align*}
&-(q-q^{-1})zJ^+(z)=:e^{\tu^+(z)}: - :e^{\tu^-(z)}:\,,
\\
&(q-q^{-1})zJ^-(z)=:e^{\tv^+(z)}: - :e^{\tv^-(z)}:\,,
\\
&\psi(z)=q^{a_0+b_0}e^{(q-q^{-1})\sum_{n>0}J^3_nz^{-n}}\,,
\\
&\varphi(z)=
q^{-a_0-b_0}e^{-(q-q^{-1})\sum_{n>0}J^3_{-n}z^{n}}\,,
\end{align*}
where
\begin{align*}
&J^3_n=q^{2n-|n|}a_n+q^{(k+2)n-\frac{k+2}{2}|n|}b_n\,,
\\
&\tu^\pm(q^{k+2}z)=-b(2|z;1)-c(2|q^{\pm1}z;0)\,, 
\\
&\tv^\pm(q^{k+2}z)=a(k+2|q^{k+2\pm k}z;-\frac{k+2}{2})
-a(k+2|q^{k}z;\frac{k+2}{2})+b(2|q^{\pm(k+2)}z;-1)+c(2|q^{\pm(k+1)}z;0)\,.
\end{align*}

In addition the Wakimoto screening current is given by
\begin{align*}
&-(q-q^{-1})zJ^S(z)=:e^{\trr^+(z)}: - :e^{\trr^-(z)}:\,,\\
&\trr^\pm(q^{k+2}z)=
-a(k+2|q^{k}z;-\frac{k+2}{2})-b(2|z;-1)-c(2|q^{\pm1}z;0)\,.
\end{align*}
\qed
\end{thm}
It is also straightforward to find two fermionic screening currents 
for $U_q\widehat{\mathfrak{sl}_2}$ 
\begin{align*}
&\tilde{\sigma}_2(z)=:e^{a(1|q^{-2}z;-\frac{k+2}{2})+b(k+2;1,2|q^{-k-2}z;-1)
+c(k+1;1,2|q^{-k-2}z;0)}: \,,
\\
&\tilde{\sigma}_3(z)=:e^{c(2|q^{-k-2}z;0)}: \,.
\end{align*}
\medskip

Now we take $k=k_1$. 
Our task is to relate the free fields $a(z),b(z),c(z)$ 
with $a_1(z),a_2(z),a_3(z)$. 
To this end we identify 
$\sigma_i(z)$ with $\tilde{\sigma}_i(q^{k+2}z)$, $i=2,3$. 
Comparison with the CFT case suggests that they commute with the Heisenberg
subalgebra spanned by $H_{1,n}$, 
i.e., $[H_{1,n},a_{2,m}] =[H_{1,n},a_{3,m}]=0$. 
This fixes $H_{1,n}$ uniquely up to a scalar multiple.
Taking the normalization into account we obtain $H_{1,n}=h_{1,n}$ where
\begin{align}
h_{1,n}&=\frac{1}{n}
\Bigl([k_1n]a_{1,n}-[(k_2+2)n]a_{2,n}-[(k_3+2)n]a_{3,n}\Bigr)
\times
\begin{cases}
-\frac{[n]}{[(k_2+2)n]} & (n>0),\\
\frac{[n]}{[(k_3+2)n]} & (n<0).\\
\end{cases}
\label{h1n-S}
\end{align}
The CFT case suggests also that we compare 
the bosonic screening current $\rho_1(z)$ 
with  the Wakimoto screening current $-J^S(q^{k+2}z)$.
Equating the oscillator part and using
the relation $J^3_n=q^{kn/2}h_{1,n}$, we can solve 
$a_n,b_n,c_n$ ($n\neq0$) for $h_{1,n},a_{2,n},a_{3,n}$ as follows. 
\begin{align}
&nq^{-kn+\frac{k+2}{2}|n|}a_n=
\frac{[(k+2)n]}{[kn]}q^{\frac{k}{2}(|n|-n)-2n} n h_{1,n} 
-\frac{[n][2n]}{[kn]}a_{2,n}
+\frac{[(k+1)n][2n]}{[kn]}a_{3,n}\,,
\label{a-S}\\
&nq^{|n|}b_n=-\frac{[2n]}{[kn]}q^{-\frac{k}{2}(n+|n|)-2n} nh_{1,n}
+\frac{[n][2n]}{[kn]} a_{2,n}
-\frac{[(k+1)n][2n]}{[kn]} a_{3,n}\,,
\label{b-S}\\
&nc_n=[2n] a_{3,n}\,.
\label{c-S}
\end{align}
From the zero mode of $\sigma_i(z)$ we have
\begin{align*}
&\frac{1}{2}Q_c=Q_{a_3}\,,\quad 
Q_a+\frac{k+2}{2}Q_b+\frac{k+1}{2}Q_c=Q_{a_2}\,,
\\
&\frac{1}{2}c_0=a_{3,0}\,,\quad 
a_0+\frac{k+2}{2}b_0+\frac{k+1}{2}c_0=a_{2,0}\,.
\end{align*}
Thus we get
$-(q-q^{-1})E_1(z)=:e^{u_1^+(z)}:-:e^{u_1^-(z)}:$, 
$(q-q^{-1})F_1(z)=:e^{v_1^+(z)}:-:e^{v_1^-(z)}:$
with 
\begin{align*}
&u^\pm_1(z)=\tu^\pm(q^{k+2}z)=Q_{u_1}+u_{1,0}\log z
\mp a_{3,0}\log q-\sum_{n\neq0}\frac{u^\pm_{1,n}}{n}z^{-n}\,,
\\
&v^\pm_1(z)=\tv^\pm(q^{k+2}z)=
-\bigl(Q_{u_1}+u_{1,0}\log z\bigr)
\pm a_{2,0}\log q -\sum_{n\neq0}\frac{v^\pm_{1,n}}{n}z^{-n}\,.
\end{align*}
Formulas for $u^\pm_{1,n}$, $v^\pm_{1,n}$ are obtained
by using \eqref{h1n-S} and \eqref{a-S}--\eqref{c-S}. 
The result is given in \eqref{u1n} and \eqref{v1n}.

It remains to express the zero mode
\begin{align*}
&Q_{u_1}= -\frac{1}{2}\bigl(Q_b+Q_c\bigr), \quad
u_{1,0}=-\frac{1}{2}(b_0+c_0)\,,
\end{align*} 
in terms of $Q_{a_i}$ and $a_{i,0}$. We have
\begin{align*}
&[u_{1,0},Q_{u_1}]=0,\quad
[u_{1,0},Q_{a_2}]=1\,, \quad [u_{1,0}, Q_{a_3}]=-1\,,\\
&[a_{2,0},Q_{u_1}]=1\,, \quad [a_{3,0},Q_{u_1}]=-1\,.
\end{align*}
These equations admit a one-parameter family of solutions
\eqref{Qu1s}, \eqref{Q0u0}--\eqref{baru} given in the main text.

For any choice of $s$, we have an identification 
\begin{align*}
&Q_a=(k_1+2)Q_{u_1}+Q_{a_2}+Q_{a_3}\,,
\quad
Q_b=-2Q_{u_1}-2Q_{a_3}\,,\quad Q_c=2Q_{a_3}\,,
\\
&a_0=(k_1+2)u_{1,0}+a_{2,0}+a_{3,0}\,,
\quad b_0=-2u_{1,0}-2a_{3,0}\,,\quad c_0=2a_{3,0}\,,
\\
&h_{1,0}=a_0+b_0=-\frac{1}{1+k_1s}\bar u_{1,0}\,.
\end{align*}

Finally the bosonic screening current
is related to the Wakimoto screening current by
\begin{align*}
&\rho_1(q^{-k-2}z)=-J^S(z)\, q^{\frac{k}{k+2}a_0}\,.
\end{align*}
The factor $a_0$ does not affect the screening property
because it commutes with $J^\pm(z),J^3(z)$. 

\bigskip

\section{Contractions}\label{sec:contraction}
In this section we summarize contractions among various currents
used in the text. 
We write
\begin{align*}
&\rho_j^\pm(z)=:e^{r_j^\pm(z)}:\,, 
\quad \tau_j^\pm(z)=:e^{t_j^\pm(z)}:\,.
\end{align*}
In the following all relations hold true under simultaneous cyclic permutations of indices
$(1,2,3)$, so we write only typical cases and omit writing ``$\&\, cycl.$''.
\bigskip

\noindent{\it Contractions for $\sigma_i(z)$}\quad
\begin{align*}
&\sigma_i(z)\sigma_i(w)=(z-w):e^{a_i(z)+a_i(w)}:\,
,\\
&\sigma_0(z)\sigma_1(w)=z^{k_1+1}\frac{(q^{-k_1}w/z;q^2)_\infty}{(q^{k_1+2}w/z;q^2)_\infty}:e^{a_0(z)+a_1(w)}:\,,
\\
&\sigma_1(w)\sigma_0(z)
=w^{k_1+1}\frac{(q^{-k_1}z/w;q^2)_\infty}{(q^{k_1+2}z/w;q^2)_\infty}
:e^{a_0(z)+a_1(w)}:\,,
\\
&\sigma_2(z)\sigma_3(w)=z^{k_1+1}\frac{(q^{-k_1}w/z;q^2)_\infty}{(q^{k_1+2}w/z;q^2)_\infty}:e^{a_2(z)+a_3(w)}:\,,
\\
&\sigma_3(w)\sigma_2(z)=w^{k_1+1}
\frac{(q^{-k_1}z/w;q^2)_\infty}{(q^{k_1+2}z/w;q^2)_\infty}
:e^{a_2(z)+a_3(w)}:\,.
\end{align*}
\bigskip

\noindent{\it Contractions for 
$\sigma_i(z)$ and $\rho^\pm_j(z),\tau^\pm_j(z)$}\quad
In all cases we have equalities (in the sense of  matrix elements)
\begin{align*}
&\rho_j^\pm(z)\sigma_i(w)=-\sigma_i(w) \rho_j^\pm(z)\,,
\quad
\tau_j^\pm(z)\sigma_i(w)=-\sigma_i(w) \tau_j^\pm(z)\,.
\end{align*}
The contractions are given by the following Table \ref{tab:art}. 
The entries in the table are the 
coefficients $X_{A,B}(z,w)$ in the product
$A(z)B(w)=X_{A,B}(z,w):A(z)B(w):$,
where $A(z)$ is taken from the first column and 
$B(w)$ from the first row. A similar convention will be used 
for other tables. 

\begin{table}[H]  
\begin{center}
\begingroup
\renewcommand{\arraystretch}{1.5}
\begin{tabular}{|c|c|c|c|c|}
\hline
& $\sigma_0(w)$ &$\sigma_{1}(w)$ &  $\sigma_{2}(w)$ &  $\sigma_{3}(w)$ \\
\hline
$\rho^\pm_{1}(z)$ & $q^{\pm(k_3+1)}z-w$ &$q^{\pm(k_2+1)}z-w$ &$(q^{\mp(k_1+1)}z-w)^{-1}$ &$(q^{\mp 1}z-w)^{-1}$ \\
\hline
$\tau^\pm_{1}(z)$ &$(q^{\mp(k_1+1)}z-w)^{-1}$ &$(q^{\mp 1}z-w)^{-1}$ & $q^{\pm(k_3+1)}z-w$ &$q^{\pm(k_2+1)}z-w$ \\
\hline
\end{tabular} 
\endgroup
\end{center}
\caption{\label{tab:art}}
\end{table}
\bigskip

\noindent{\it Contractions for 
$\rho_i^\pm(z)$ and $\tau_i^\pm(z)$ }\quad
\begin{align*}
&\rho^{\epsilon_1}_1(z)\rho^{\epsilon_2}_1(w)=z^{\frac{2}{k_1+2}}q^{-\epsilon_1}
\frac{(p_1q^{-2}w/z;p_1)_\infty}{(q^2w/z;p_1)_\infty}f_{\epsilon_1\epsilon_2}(w/z)
\, :e^{r_1^{\epsilon_1}(z)+r_1^{\epsilon_2}(w)}:\,,
\\
&\tau^{\epsilon_1}_1(z)\tau^{\epsilon_2}_1(w)=z^{\frac{2}{k_1+2}}q^{-\epsilon_1}
\frac{(p_1q^{-2}w/z;p_1)_\infty}{(q^2w/z;p_1)_\infty}f_{\epsilon_1\epsilon_2}(w/z)
\, :e^{t_1^{\epsilon_1}(z)+t_1^{\epsilon_2}(w)}:\,,
\\
&\tau^{\epsilon_1}_1(z)\rho^{\epsilon_2}_1(w)=z^{-\frac{2}{k_1+2}}q^{\epsilon_1}
\frac{(q^{-k_2}w/z,p_1q^{k_2+2}w/z;p_1)_\infty}
{(q^{-k_2-2}w/z,p_1q^{k_2}w/z;p_1)_\infty}g_{\epsilon_1\epsilon_2}(w/z)
\, :e^{t_1^{\epsilon_1}(z)+r_1^{\epsilon_2}(w)}:
\,,\\
&\rho^{\epsilon_1}_1(z)\tau^{\epsilon_2}_1(w)=z^{-\frac{2}{k_1+2}}q^{\epsilon_1}
\frac{(q^{-k_2}w/z,p_1q^{k_2+2}w/z;p_1)_\infty}
{(q^{-k_2-2}w/z,p_1q^{k_2}w/z;p_1)_\infty}g_{\epsilon_1\epsilon_2}(w/z)
\, :e^{r_1^{\epsilon_1}(z)+t_1^{\epsilon_2}(w)}:
\,,\\
\end{align*}
where
\begin{align*}
&f_{\pm,\pm}(z)=1-z\,,\quad f_{\pm,\mp}(z)=1-q^{\pm 2}z\,,
\\
&g_{\pm,\pm}(z)=1\,,\quad g_{\pm,\mp}(z)=\frac{1-q^{\mp(k_2+2)}z}{1-q^{\mp k_2}z}\,.
\end{align*}
\bigskip

\noindent{\it Contractions for $\rho_i^\pm(z)$ and $\tau_j^\pm(z)$ 
($i\neq j$)}\quad

Non-trivial contractions are as follows, where double-sign correspondence 
is implied (notice the upper index in $\tau^\mp_1$).
\begin{table}[H]  
\begin{center}
\begingroup
\renewcommand{\arraystretch}{2}
\begin{tabular}{|c|c|c|c|c|}
\hline
& $\rho^\pm_1(w)$ &$\rho^\pm_{2}(w)$ &  $\tau^\mp_{1}(w)$ &  $\tau^\pm_{2}(w)$ \\
\hline
$\rho^\pm_{1}(z)$ & $\ast$ &$\ds{q^{\mp1}\frac{z-q^{\mp k_2}w}{z-q^{\mp (k_2+2)}w}}$ & $\ast$ &
$\ds{q^{\pm1}\frac{z-q^{\pm k_1}w}{z-q^{\pm (k_1+2)}w}}$\\
\hline
$\rho^\pm_{2}(z)$ &$\ds{q^{\pm1}\frac{z-q^{\pm k_2}w}{z-q^{\pm (k_2+2)}w}}$&
$\ast$ &
$\ds{q^{\mp1}\frac{z-q^{\pm 2}w}{z-w}}$
&$\ast$ \\
\hline
$\ds{\tau^\mp_{1}(z)$ &$\ast$ & $\ds{q^{\pm1}\frac{z-q^{\mp 2}w}{z-w}}}$&
$\ast$ &
$\ds{q^{\mp1}\frac{z-q^{\pm (k_3+2)}w}{z-q^{\pm k_3}w}}$\\
\hline
$\tau^\pm_{2}(z)$ &$\ds{q^{\mp1}\frac{z-q^{\mp k_1}w}{z-q^{\mp (k_1+2)}w}}$
&$\ast$ &$\ds{q^{\pm1}\frac{z-q^{\mp (k_3+2)}w}{z-q^{\mp k_3}w}}$&
$\ast$ \\
\hline
\end{tabular} 
\endgroup
\end{center}
\caption{\label{tab:rho-tau2}}
\end{table}
Asterisks indicate cases with the same subindex, 
which have been treated before.
Contractions not shown in the Table (up to permutation of indices $1,2,3$) 
are trivial;
for example we have $\rho_1^\pm(z)\tau_2^\mp(w)=:e^{r_1^\pm(z)+t_2^\mp(w)}:$. 
\bigskip

\noindent{\it Contractions for $\rho_1^\pm(z)$, 
$\sigma_2(z)$, $\sigma_3(z)$ with generators of $\E_{2,1}$}\quad
We set
\begin{align*}
U^\pm_1(z)= :e^{u^\pm_1(z)}:\,,\quad 
V^\pm_1(z)= :e^{v^\pm_1(z)}:\,,\quad 
\Uh^\pm_1(z)= :e^{\uh^\pm_1(z)}:\,,\quad 
\Vh^\pm_1(z)= :e^{\vh^\pm_1(z)}:\,.
\end{align*}

We have equalities for $i=2,3$
\begin{align*}
&U^{\varepsilon_1}(z)\rho^{\varepsilon_2}_1(w)=
\rho^{\varepsilon_2}_1(w)U^{\varepsilon_1}(z)\,,
\quad
V^{\varepsilon_1}(z)\rho^{\varepsilon_2}_1(w)=
\rho^{\varepsilon_2}_1(w)V^{\varepsilon_1}(z)\,,
\\
&\Uh^{\varepsilon_1}(z)\rho^{\varepsilon_2}_1(w)=
\rho^{\varepsilon_2}_1(w)\Uh^{\varepsilon_1}(z)\,,
\quad
\Vh^{\varepsilon_1}(z)\rho^{\varepsilon_2}_1(w)=
\rho^{\varepsilon_2}_1(w)\Vh^{\varepsilon_1}(z)\,,
\\
&U^\pm(z)\sigma_i(w)=-\sigma_i(w) U^\pm(z)\,,\quad
V^\pm(z)\sigma_i(w)=-\sigma_i(w) V^\pm(z)\,,
\\
&\Uh^\pm(z)\sigma_i(w)=-\sigma_i(w) \Uh^\pm(z)\,,\quad
\Vh^\pm(z)\sigma_i(w)=-\sigma_i(w) \Vh^\pm(z)\,.
\end{align*}

\begin{table}[H]  
\begin{center}
\begingroup
\renewcommand{\arraystretch}{1.5}
\begin{tabular}{|c|c|c|c|c|}
\hline
& $\rho_1^\pm(w)$ &$\rho_1^\mp(w)$ &$\sigma_{2}(w)$ & $\sigma_{3}(w)$ \\
\hline
$U_1^\pm(z)$ & $\ds{q^{\pm1}\frac{1-q^{\mp2}w/z}{1-w/z}}$& $q^{\pm1}$
& $q^{\mp(k_1+1)}z-w$ & $\ds{\frac{1}{q^{\pm 1}z-w}}$ \\
\hline
$V_1^\pm(z)$ & $\ds{q^{\mp1}\frac{1-q^{\mp k_1}w/z}{1-q^{\mp(k_1+2)}w/z}}$
& $q^{\mp 1}$ 
& $\ds{\frac{1}{q^{\mp 1}z-w}}$ & $q^{\pm(k_1+1)}z-w$\\
\hline
$\Uh_1^\mp(z)$ & $\ds{q^{\mp 1}\frac{1-q^{\mp k_1}w/z}{1-q^{\mp(k_1+2)}w/z}}$
& $q^{\mp1}$ 
& $\ds{\frac{1}{q^{\mp 1}z-w}}$ & $q^{\pm(k_1+1)}z-w$\\
\hline
$\Vh_1^\mp(z)$ & $\ds{q^{\pm1}\frac{1-q^{\mp2}w/z}{1-w/z}}$
& $q^{\pm1}$
& $q^{\mp(k_1+1)}z-w$ & $\ds{\frac{1}{q^{\pm 1}z-w}}$\\
\hline
\end{tabular} 
\endgroup
\end{center}
\caption{\label{tab:UVsigma}}
\end{table}
\bigskip

{\bf Acknowledgments.}
The research of BF is supported by 
the Russian Science Foundation grant project 16-11-10316. 
MJ is partially supported by 
JSPS KAKENHI Grant Number JP16K05183. 
EM is partially supported by a grant from the Simons Foundation  
\#353831.

The authors would like to thank Kyoto University
for hospitality during their visits when this work was started.

\end{document}